\documentclass{amsart}
\usepackage{graphicx}
\usepackage[all]{xy}
\usepackage[ansinew]{inputenc}
\usepackage{amsmath}
\usepackage{amscd}
\usepackage{amssymb}
\usepackage{latexsym}
\usepackage[active]{srcltx}
\usepackage{mathrsfs}
\usepackage{color}
\usepackage{enumerate}
\usepackage{hyperref}
\usepackage{tipa}

\newtheorem{thm}{Theorem}[section]

\theoremstyle{definition}
\newtheorem{cor}[thm]{Corollary}
\newtheorem{lem}[thm]{Lemma}
\newtheorem{prop}[thm]{Proposition}
\newtheorem{defn}[thm]{Definition}

\newtheorem{case}{Case}
\newtheorem*{remark}{Remark}

\numberwithin{equation}{section}

\newcommand{\N}{\mathbb{N}}
\newcommand{\Z}{\mathbb{Z}}
\newcommand{\Q}{\mathbb{Q}}
\newcommand{\R}{\mathbb{R}}

\newcommand{\tp}{\operatorname{tp}}

\newcommand{\dcl}{\operatorname{dcl}}

\newcommand{\eq}{\operatorname{eq}}

\newcommand{\Cal}{\mathcal}

\newcommand{\st}{\operatorname{st}}

\newcommand{\proves}{\vdash}

\def \<{\langle}
\def \>{\rangle}
\def \hat {\widehat}
\def \((  {(\!(}
\def \)) {)\!)}

\usepackage{xcolor}

\begin{document}

\title[Distal and non-Distal Behavior in Pairs]{Distal and non-Distal Behavior in Pairs}

\author[T. Nell]{Travis Nell}
\address
{Department of Mathematics\\University of Illinois at Urbana-Champaign\\1409 West Green Street\\Urbana, IL 61801}
\email{tnell2@illinois.edu}
\urladdr{http://www.math.uiuc.edu/\textasciitilde tnell2}

\subjclass[2010]{Primary 03C64  Secondary 03C45}

\date{\today}

\begin{abstract}

The aim of this work is an analysis of distal and non-distal behavior in dense pairs of o-minimal structures. A characterization of distal types is given through orthogonality to a generic type in $\Cal M^{\eq}$, non-distality is geometrically analyzed through Keisler measures, and a distal expansion for the case of pairs of ordered vector spaces is computed. 


\end{abstract}

\maketitle

\section{Introduction}\label{Introduction}

Simon, in \cite{Distal}, isolated a subclass of NIP theories, called \emph{distal theories}, encompassing all those theories that can be considered purely unstable. Every o-minimal theory, or even any ordered dp-minimal theory is distal. However, the definability of a linear order does not guarantee distality of an NIP theory. In \cite{DNDP}, we established that various well-known NIP expansions of o-minimal theories by dense and codense sets fail distality. Among these are the theory of dense pairs of o-minimal structures (as studied by van den Dries in \cite{DensePairs}), the theory of the real field expanded by a predicate for dense subgroup of $\R_{>0}$ with the Mann property (see van den Dries and G\"unaydin \cite{MannProperty}), and the theory of an o-minimal structure expanded by a dense $\dcl$-independent subset (see Dolich, Miller and Steinhorn \cite{IndependentPairs}).\newline

\noindent Let $\Cal A=(A,<,\dots)$ be an o-minimal structure expanding an ordered group and let $B\subseteq A$. Denote by $\Cal M$ the pair $(\Cal A,B)$ and assume that it is a model of one of the theories mentioned above. Immediately after our work in \cite{DNDP}, Simon raised the following questions:\newline

\noindent \textbf{Question 1.} Does $\Cal M$ admit a distal expansion? That means, is there an expansion of $\Cal M$ that has a distal theory? \newline
\noindent \textbf{Question 2.} Is there a family of generically stable types in $\Cal M^{\eq}$ such that any invariant type orthogonal to these types is distal?\newline

\noindent While distality is in general not preserved under reducts, there are combinatorial consequences of distality established by Chernikov and Starchenko in \cite{RegLemma} that are preserved under reducts. Therefore the property of having a distal expansion is desirable in and of itself, and explains the relevance of Question 1. Note by Section 6 of \cite{RegLemma}, for $p>0$ the theory of algebraically closed fields of characteristic $p$ is NIP, but does not admit a distal expansion. By \cite{WildCore}, there are even NIP expansions of the real field that do not admit a distal expansion.\newline

\noindent Question 2 is an attempt to better understand the non-distality of these structures. Note that every distal invariant type is orthogonal to every generically stable type \cite{SimonBook}, however there are examples where a non-distal type is orthogonal to every generically stable type. When $B$ is a dense $\dcl$-independent subset of $\Cal A$, the pair $(\Cal A,B)$ has elimination of imaginaries by \cite{IndependentPairs}. Therefore in this example, there are no nontrivial generically stable types even in $\Cal M^{\eq}$. Thus we already know that in this situation Question 2 has a negative answer.\newline

\noindent In the current paper we answer Question 1 positively in the case that $\Cal A$ is an ordered vector space over an arbitrary ordered field $F$, and $B$ is proper dense $F$-subspace of $\Cal A$. In such a situtation, there is a natural group quotient $A/B$, and we will show that this carries the pure structure of an unordered vector space over $F$. This naturally evidences the stable, non-distal behavior of the pair $(\Cal A,B)$ we found in \cite{DNDP}. To reach a distal expansion, we place an ordered $F$-vector space structure upon this quotient. Along the way to the previous result, we also prove that the pair $(\Cal A,A/B;+,0,1,<,B)$ admits a weak elimination of imaginaries. While not necessary for the main thrust of the distal expansion, this does seemingly illustrate the only obstruction to distality is the missing order on the imaginary sort $A/B$. 
\newline

\noindent It is natural to wonder whether our methods can be used to handle all three classes of non-distal NIP theories of expansion of o-minimal structures studied in \cite{DNDP}. We believe that all these theories admit a distal expansion. However, it is unclear how far the idea of ordering the classes of all definable equivalence relations can be pushed. Again, note that when $B$ is a dense $\dcl$-independent subset of $\Cal A$, the pair $(\Cal A,B)$ has elimination of imaginaries by \cite{IndependentPairs}. Therefore, in this situation for every definable equivalence relation there is a definable order on its classes, but the structure is nevertheless not distal.\newline

\noindent We furthermore give a complete answer to Question 2 in the case that $\Cal A$ is an o-minimal expansion of an ordered group, and $B$ is a dense proper elementary substructure. This is done by finding a generically stable type in the sort of the quotient group $A/B$, which completely characterizes the distality. This type serves as a singleton family answering Question 2. We also give a more geometric proof of the non-distality of the original pair, by demonstrating that the Keisler measure induced by the Lebesgue measure is generically stable, but not smooth.  

\subsection*{Acknowledgements} We thank Pierre Simon for asking Questions 1 and 2, along with Erik Walsberg for his thoughts on the approach in Section \ref{GenericExistence}, and Philipp Hieronymi for his guidance throughout this project. The author was partially supported by NSF grant DMS-1654725.

\section{Preliminaries}\label{Preliminaries}

 We begin by setting some notation. For a multisorted structure $\Cal M$ with sorts $(s_i)_{i\in I}$, $i_1,\ldots, i_n \in I$, and $s = (s_{i_1},\ldots, s_{i_n})$, we say $M_s = \{(a_1,\ldots,a_n):a_j \in M_{s_{i_j}}\}$. That is, tuples where the $j$-th element lies in the sort $s_{i_j}$. Note that we distinguish between structures and underlying sets through use of italics. \newline

\noindent Throughout this section fix a first order theory $T$ in a multi-sorted language $\Cal L$. Given an $\Cal L$-formula $\varphi(x,y)$, a model $\Cal M$ of $T$, tuples of sorts $s$ and $t$, $\Cal N \succeq \Cal M$, and $b \in N_t$, we write $\varphi(M,b)$ to denote the collection of all $a \in M_s$ such that $N \models \varphi(a,b)$.  Furthermore for $a \in M_s$, and some definable $R\subseteq M_s \times M_t$, we set $R(a) = \{b \in M_t: (a,b) \in R\}$. We refer to this as the fiber of $R$ above $a$. Similarly for $b \in M_t$ we call $R(b)=\{a \in M_s:(a,b) \in R\}$ the fiber of $R$ above $b$.  For a definable $S \subseteq M$ and a definable function $f:S\to M_t$, the graph of $f$, written $gr(f)$, is $\{(a,b) \in S \times M_t:f(a)=b\}$. If $p(x)$ is a type over some $A \subseteq M$ with $x$ varying over some tuple of sorts, we write $p \in S_x(A)$.\newline

\noindent For ease of notation, we will often work over a large, highly saturated model $\Cal U$. Whenever mentioned, all other sets and models are to be assumed small in cardinality relative to the saturation of $\Cal U$. Notice also that all such definitions can be localized to mention only small models, at the cost of quantifier complexity of the definitions. \newline

\noindent Unless otherwise specified, a set is definable if it is definable with parameters.

\begin{defn}

Let $\Cal M \models T$, $s=(s_1,\ldots,s_n)$ and $t=(t_1,\ldots,t_m)$ be tuples of sorts, and $A \subseteq M_s$ be definable. Then for any definable or type-definable $S \subseteq M_t$, we say $S$ is $A$-\emph{small} if there is $\ell \in \N$ and a definable $f:M_s^\ell \to M_t$ such that $S\subseteq f(A^\ell)$. If $S$ is not $A$-small, then it is $A$-\emph{large}.

\end{defn}

\noindent We now remind the reader of certain definitions involving types that appear often in the context of NIP theories.

\begin{defn}

Let $t=(t_1,\ldots,t_n)$ be a tuple of sorts, and $p(x)\in S_t(U_t)$. We say $p(x)$ is \emph{generically stable} if there is a small $A \subset U$ such that:

\begin{enumerate}

\item ($p$ is definable over $A$.) For each $\varphi(x,y)$, the set $\{b \in U_y:\varphi(x,b) \in p(x)\}$ is $A$-definable.  

\item ($p$ is finitely satisfiable in $A$.) For each $\varphi(x,b) \in p(x)$, there is $a \in A$ such that $\Cal U \models \varphi(a,b)$. 

\end{enumerate}

\end{defn}

\begin{remark}

Notice that definability of a type $p(x) \in S_x(M)$ gives a canonical extension $p(x)|N$ to any $\Cal N \succeq \Cal M$. Suppose that $p(x)$ is defined over $A$. Then for any $b \in N$, $\varphi(x,b) \in p(x)|N$ if and only if there is $b' \in M$ with $\tp(b|A)=\tp(b'|A)$ such that $\varphi(x,b')\in p(x)$. 

\end{remark}

\begin{defn}

Let $t=(t_1,\ldots,t_n)$ be a tuple of sorts, and $p(x)\in S_t(U_t)$. We say $p(x) \in S_t(U_t)$ is \emph{invariant} if there is some $A \subset U_t$ that is small (in cardinality) such that for each $a,a' \in U_t$ with $\tp(a|A)=\tp(a'|A)$, $\varphi(x,a)\in p(x)$ if and only if $\varphi(x,a') \in p(x)$. Such a type is then called \emph{$A$-invariant}. 

\end{defn}

\begin{remark}
The previous definition also works in the case where $\Cal U$ is replaced with an arbitrary model of $T$. Furthermore, we get an analogy of canonical extensions of definable types in this case. Suppose $A \subset M$, and $\Cal M$ is $|A|^+$-saturated. Then consider $\Cal N \succeq \Cal M$. For any $a \in N$ we may define $p(x)|N$ by $\varphi(x,a) \in p(x)|N$ if and only if there is $a' \in M$ with $\tp(a|A)=\tp(a'|A)$ and $\varphi(x,a') \in p(x)$. 
\end{remark}

\noindent A key concept in this work is that of distality. We now provide a definition of a distal partial type. 

\begin{defn}

Let $A \subset U$, $x$ be a variable ranging over a finite tuple of sorts, and $\zeta(x)$ a partial $x$-type over $A$. Let $I_1$ and $I_2$ be infinite linear orders without endpoints, and $(c)$ be a one-element linear order. We say that $\zeta(x)$ is \emph{distal} if for each indiscernible sequence $(a_i)_{i\in I_1+(c)+I_2}$  in $\zeta(x)$ and each $b \in U$, $$(a_i)_{i \in I_1+I_2} \text{ is } b-\text{indiscernible} \Leftrightarrow (a_i)_{i \in I_1+(c)+I_2} \text{ is } b-\text{indiscernible}$$ We say a theory $T$ is \emph{distal} if for each variable $x$ in a finite collection of sorts, the partial type $\{x=x\}$ is distal.

\end{defn}

\begin{remark}
In the previous definition, one may replace $I_1$ and $I_2$ with any concrete example of an infinite linear order without endpoints (for example, $\Q$ or $\Z$). Furthermore, by Theorem 2.28 of \cite{Distal}, one can repace ``for each variable $x$ in a finite collection of sorts'' with ``for each sort $s$ and $x$ varying over $U_s$''.
\end{remark}

\noindent We now give a definition of what it means for two types to be weakly orthogonal. 

\begin{defn}

Let $x$ and $y$ be variables ranging over a finite tuple of sorts, $A \subset U$, $p(x)\in S_x(A)$ and $q(y)\in S_y(A)$. Then we say that $p(x)$ and $q(y)$ are \emph{weakly orthogonal} if $p(x) \cup q(y)$ determines a unique extension to a complete type over $A$ in variables $x$ and $y$. 

\end{defn}

\noindent In the discussion of generically stable types and distal types, we will often work with induced structure on a subset. We now fix what we mean by this notion.

\begin{defn}

Let $t$ be a finite tuple of sorts, $\Cal M\models T$, and $S \subseteq M_t$ be definable.  Then by the \emph{induced structure from $\Cal M$ on $S$} we mean the $\Cal L'$-structure whose underlying set is $S$, where $\Cal L'$ has a relation symbol for each subset of  $S^n$ definable in $\Cal M$ (with parameters from $M$), interpreted in the natural way. 

\end{defn}

\section{Generically Stable Types}\label{GenericExistence}

In this section we work in a more general setting than needed for our main results, that of structures obeying the tameness conditions of \cite{PairStructure}. In less generality most structural lemmas were known previously to this paper, but we prefer to use this reference. All examples in this paper meet the technical assumptions of \cite{PairStructure}\newline

\noindent Let $\tilde{\Cal M}=(\Cal M, P)$ be an expansion of an o-minimal $\Cal L$-structure $\Cal M $ by a dense subset $P$ satisfying the tameness conditions in \cite{PairStructure}. By $M$, we refer to the underlying set of $\Cal M$, which we shall refer to as the home sort. Let $T_P$ be the $\Cal L_P$-theory of $\Cal M$. By small or large, we mean that a definable set is $P$-small or $P$-large.  \newline

\noindent We now introduce the results from \cite{PairStructure} that we need for Sections \ref{GenericExistence} through Section \ref{Measures}. In certain cases, the full strength of the cited results is unnecessary, so we have restated to match the usage in this section. The following is a consequence of Lemma 3.3 from \cite{PairStructure}.

\begin{lem}\label{PairDecomp}

Let $(X_t)_{t\in M^\ell}$ be an $A$-definable family of subsets of $M$. Then there is $m \in \N$, such that for $i =\{1,\ldots m\}$ there are

\begin{itemize}

\item an $A$-definable family $(V_{i,t})_{t \in M^\ell}$ of small subsets of $M$
\item an $A$-definable function $a_i:M^\ell \to M \cup \{\infty\}$

\end{itemize}

such that for $t\in M^\ell$, 

\begin{enumerate}
\item $-\infty = a_0(t)\le \ldots \le a_m(t)=\infty$ is a decomposition of $M$, and
\item one of the following holds:

\begin{itemize}
\item $[a_{i-1}(t),a_i(t)]\cap X_t= [a_{i-1}(t),a_{i}(t)]\cap V_{i,t}$ or 
\item $[a_{i-1}(t),a_i(t)]\cap X_t =[a_{i-1}(t),a_i(t)]\setminus V_{i,t}$. 
\end{itemize}
\end{enumerate}
\end{lem}

\begin{remark}
The authors in \cite{PairStructure} follow by using this lemma to show in Remark 3.4 that $\{t \in M^\ell:X_t \text{ is small}\}$ is $\emptyset$-definable. 
\end{remark}

\noindent The following property of small sets will also be needed. This is a consequence of Lemma 4.29 of \cite{PairStructure}.

\begin{lem}\label{SmallProd}

Let $X \subset M \times M$ be definable. Suppose $\pi(X)$, the projection onto the first coordinate, is small, and that for each $t \in \pi(X)$, that $\{y:(t,y) \in X\}$ is small. Then $X$ is small.  

\end{lem}

\noindent One may be concerned with the difference in definition of small between this work and \cite{PairStructure}. However, Lemma 3.11 of \cite{PairStructure} shows the equivalence between these two definitions.\newline

\noindent Throughout this section, we fix $-\infty \le a < b\le \infty$ in $M\cup\{\pm\infty\}$. Let $E$ be a $\emptyset$-definable equivalence relation on $(a,b)$ with small, dense classes. When it is important to specify in which model an interval is defined, we shall write $(a,b)_M$ to mean $\{x \in M: a<x<b\}$, while for the predicate $P$, we shall write $P(M)$ for $\{x \in M: \Cal M \models P(x)\}$. Notice that then there is an imaginary sort $(a,b)/E$. In this section, we establish the existence of a generically stable type in this sort. Notationally, let $\pi: (a,b) \to (a,b)/E$ be the natural quotient map, and $[x]_E$ be the $E$-class of $x$. We say a set $S \subseteq (a,b)$ is \emph{$E$-invariant} if it is a union of $E$-classes. Recall that by $\Cal U$, we mean a highly saturated model of $T_P$. \newline

\noindent The reader may consider the following example in this section and its successor. Consider $\Cal M = (\R;+,\times)$, the real field, and $P=\Q^{ra}$, the real algebraic numbers. In this case we may consider the $\Cal L_P$-definable equivalence relation $a E b$ if $a-b \in \Q^{ra}$. Notice that each $E$-equivalence class is dense and $\Q^{ra}$-small. This section then gives a generically stable type in the quotient sort of $\R/E$. 

\begin{thm}

Let $S \subseteq (a,b)$ be definable. If $S$ is large and $E$-invariant, then $(a,b) \setminus S$ is small.  

\end{thm}

\begin{proof}

We begin by applying Lemma \ref{PairDecomp} to the singleton family $(S)$. This yields $m \in \N$, a decomposition $a \le a_0\le \ldots \le a_m\le b$ of $(a,b)_M$, and small sets $V_i$ for $i \in \{1,\ldots m\}$ with the property that for each such $i$, $S \cap [a_{i-1},a_i] = V_i$ or $S \cap [a_{i-1},a_i] = [a_{i-1},a_i] \setminus V_i$. 

As $S$ is large, it is not a finite union of small sets. Therefore there must be some $i$ such that $S \cap [a_{i-1},a_i] = [a_{i-1},a_i] \setminus V_i$. Consider $A = \bigcup_{x \in V_i}[x]_E$. Now $\{(x,y):x \in V_i, y \in [x]_E\}$ is small by Lemma \ref{SmallProd}, and definably surjects onto $A$. Therefore $A$ is small. 

We now show that $((a,b) \setminus S) \subseteq A$. Suppose there is $x \in (a,b)\setminus S$ that is not in $A$. Then $[x]_E \cap V_i = \emptyset$. As $[x]_E$ is dense in $(a,b)$, there is $y \in [x]_E \cap [a_{i-1},a_i]$. Then $y \not \in V_i$, so $y \in S \cap [a_{i-1},a_i]$. As $S$ is $E$-invariant, $x \in S$, contradicting the choice of $x$. Therefore $((a,b)\setminus S) \subseteq A$. As $A$ is small, $(a,b)\setminus S$ is small.

\end{proof}

\begin{remark}
By the remark following Definition 2.1 in \cite{PairStructure}, if $S\subset (a,b)$ is small, then $(a,b) \setminus S$ is large. As also the union of finitely many small sets is small, the large $E$-invariant sets form an ultrafilter on the boolean algebra of $E$-invariant definable subsets of $(a,b)$. Thus, this is a type in the quotient sort $(a,b)/E$. Denote this type $q(y)$. Notice that for an $\Cal L_P$-formula $\varphi(y)$, we have $\varphi(y) \in q(y)$ if and only if $\pi^{-1}(\varphi(M))$ is large. 
\end{remark}

\noindent Let $S \subseteq M^{m+n}$ be $\emptyset$-definable. Recall that $\{t \in M^n: S(t) \text{ is \emph{small}}\}$ is $\emptyset$-definable. Therefore we have the following:

\begin{lem}
The type $q(y)$ is $\emptyset$-definable.
\end{lem}

\noindent To further show that $q(y)$ is generically stable, we now need to show finite satisfiability.

\begin{thm} \label{TypeFinSat}
The type $q(y)$ is finitely satisfiable. 
\end{thm}

\begin{proof}

Let $S\subseteq (a,b)_U$ be $E$-invariant, large, and $c$-definable for some $c \in U^n$. Assume for sake of contradiction that $S \cap M = \emptyset$. Then $(a,b)_M$ is contained in $(a,b)_U \setminus S$. 
As $(a,b)_U \setminus S$ is small, there is $m \in \N$ and a $c$-definable function $f:U^m \to U$ such that $((a,b)_U\setminus S)\subseteq f(P(U)^m)$. Therefore $(a,b)_M \subseteq \dcl_{\Cal L}(P(U)c)$. Fix $k=\dim(c/P(U))$, the $\dcl_\Cal L$-dimension of $c$ over $P(U)$. Then $\dim((a,b)_M/P(U))\le k$. As $P$ is small, there are $a_1,\ldots,a_{k+1}$ in $(a,b)_M$, $\dcl_{\Cal L}$-independent over $P(M)$. As $\tilde{\Cal M} \prec \Cal U$, $a_1,\ldots, a_{k+1}$ are $\dcl_{\Cal L}$-independent over $P(U)$. This contradicts $\dim((a,b)_M/P(U))\le k$. Therefore $S \cap M \not = \emptyset$, and $q(y)$ is finitely satisfiable. 

\end{proof}

\section{Distal Types}

We continue with $\tilde {\Cal M}$, $a$, $b$, $E$, and $q(y)$ as in Section \ref{GenericExistence}. Recall that $q(y)$ is the generic type on $(a,b)/E$. Let $\tilde{\Cal N} \succeq \tilde{\Cal M}$ be $|M|^+$-saturated. Most results in this section hold in the same generality as in Section \ref{GenericExistence}, but some will require further assumptions on $T_P$. These additional assumptions on $P(M)$ will be stated in the relevant theorems. Throughout this section, we assume that $\Cal L$ contains symbols for all $\Cal L$-definable functions. Thus for any $S\subseteq M$, we have that $\langle S \rangle = \{f(s): f \in \Cal L, s \in S\}$ is an $\Cal L$-substructure of $\Cal M$. 

\begin{lem}\label{SmallSurj}
Let $x = (x_1,\ldots, x_n)$ be a variable ranging over $M^n$. Let $p(x) \in S_x(M)$ be such that for each $i \in \{1,\ldots,n\}$, $p(x)\proves a<x_i<b$. Then

\begin{enumerate}

\item If $p(x)$ is weakly orthogonal to $q(y)$, the type-definable set $p(N)$ is small.

\item If $P(N)$ is a dense $\Cal L$-elementary substructure of $\Cal N$ and $p(N)$ is small, then $p(x)$ is weakly orthogonal to $q(y)$.

\end{enumerate}
\end{lem}

\noindent Before continuing with the proof, we recall some facts from Section 2 of \cite{DensePairs} concerning a back and forth system between substructures of dense pairs of o-minimal structures. In our notation this is the case where the interpretation of the predicate $P$ is a dense elementary substructure of $\Cal M$. Unless explicitly stated, we will \emph{not} necessarily be working in this situation, but some results only hold in this specific case. We denote the back and forth system from \cite{DensePairs} by $\Gamma$, which we now translate into our notations. In the case of substructures of $\Cal N$, $\Gamma$ consists of isomorphisms $j:\Cal A_1 \cong \Cal A_2$, where for each $i \in \{1,2\}$, $|A_i|<|M|^+$, along with the property that $A_i$ and $P(N)$ are \emph{free} over $P(A_i)$. That is, that for any $S \subseteq A_i$:
$$S \text { is } \dcl_{\Cal L}\text{-independent over } P(A_i) \Rightarrow S \text{ is } \dcl_{\Cal L}\text{-independent over } P(N)$$

\begin{proof}

We first prove (1), by showing that if $c=(c_1,\ldots, c_n) \models p$, then for each $i \in \{1,\ldots, n\}$, $c_i$ lies in a small $E$-invariant $M$-definable unary set. Suppose not. Then without loss of generality we may assume $i=1$, and thus $\tp(\pi(c_1)|M) = q(y)$, so $(c,\pi(c_1)) \models p(x) \cup q(y)$. Then consider $d \models q(y)|Mc\pi(c_1)$, the definable extension of $q(y)$ to $M\cup\{c,\pi(c_1)\}$. As $[c_1]_E$ is small, we have that $d \neq \pi(c_1)$. As $(c,d) \models p(x) \cup q(y)$, we have that both $\pi(x_1) = y$ and $\pi(x_1) \neq y$ are consistent with $p(x) \cup q(y)$. Thus $p(x)$ is not weakly orthogonal to $q(y)$, contradicting our assumption. Therefore for each $i\in \{1,\ldots n\}$, $c_i$ lies in a small $E$-invariant $M$-definable set. 

We now prove (2). So we now add the further assumption that $P(N)$ is a dense, proper elementary $\Cal L$-substructure of $\Cal N$. Now suppose there is a small $M$-definable set $S\subset N^{n}$ such that if $c=(c_1,\ldots,c_n) \models p(x)$, then $c \in S$. We now show that $p(x)$ is weakly orthogonal to $q(y)$. Let $d_1,d_2 \in N$ be realizations of $p(x)$, and $e_1,e_2 \in \Cal N^{\eq}$ be realizations of $q(y)$. It suffices to find an automorphism of $\Cal N^{\eq}$ sending $(d_1,e_1)$ to $(d_2,e_2)$. Notice that this is equivalent to finding $e_1',e_2' \in N$ with $\pi(e_i')=e_i$ and an automorphism sending $(d_1,e_1')$ to $(d_2,e_2')$. To do this, we shall find $j:\Cal A_1 \cong \Cal A_2$ in $\Gamma$ such that $d_1,e_1'\in \Cal A_1$, $d_2,e_2' \in \Cal A_2$, $j(d_1)=d_2$, and $j(e_1')=e_2'$. 

Now since $d_1,d_2$ lie in a small $M$-definable set, there is an $M$-definable function $f:N^\ell \to N^n$ such that $d_1,d_2 \in f(P(N))$. Let $h_1 \in P(N)^\ell$ be such that $f(h_1)=d_1$. As $d_1$ and $d_2$ have the same $\Cal L_P$-type over $M$, there is $g \in \operatorname{Aut}_{\Cal L_P}(\Cal N|M)$ such that $g(d_1)=d_2$. Let $h_2 = g(h_1)$. We then set $\Cal A_i'=\langle Mh_i\rangle$ for $i=1,2$. Notice then that $f(h_2)=d_2$, and $\Cal A_1' \cong \Cal A_2'$ via $g$. Furthermore, as only members of $P(N)$ were added, this isomorphism is in $\Gamma$.

Recall that $\pi^{-1}(e_i)$ lies in no small $M$-definable set. Then as $\Cal M \preceq \Cal N$, $\pi^{-1}(e_i) \cap \langle M P(N) \rangle = \emptyset$. Choose $e_1' \in N \cap \pi^{-1}(e_1)$. By saturation of $\Cal N$ and density of $\pi^{-1}(e_2)$, there is $e_2' \in N \cap \pi^{-1}(e_2)$ realizing $\tp_{\Cal L}(e_1'|Mh_1h_2)$. Set $\Cal A_i = \langle A_i'e_i'\rangle$ for $i=1,2$. The map determined by $e_1'\mapsto e_2'$ extending the isomorphism $\Cal A_1'\cong \Cal A_2'$ is an isomorphism between $\Cal A_1$ and $\Cal A_2$. Recalling that $\pi^{-1}(e_i) \cap \langle M P(N) \rangle = \emptyset$, the isomorphism $\Cal A_1\cong \Cal A_2$ is in $\Gamma$. Thus there is an automorphism of $\Cal N^{\eq}$ such that $d_1 \mapsto d_2$ and $e_1 \mapsto e_2$. Therefore $p(x) \cup q(y)$ determines a complete type, so $p(x)$ is weakly orthogonal to $q(y)$. 
\end{proof}

\begin{lem}\label{InducedStructure}

Suppose that $S \subset M$ is small, and the induced structure on $P(M)$ from $\tilde{\Cal M}$ is distal. Then the induced structure on $S$ from $\tilde{\Cal M}$ is distal. 

\end{lem}

\begin{proof}

As $S$ is small, there is an $n \in \N$ and an $M$-definable function $f:M^n \to M$ such that $S \subseteq f(P(M)^n)$. Let us now consider the equivalence relation $\equiv$ on $P(M)^n$ given by $x \equiv y \Leftrightarrow f(x)=f(y)$. Notice that as $f$ is $M$-definable, this relation is a part of the $\tilde{\Cal M}$-induced structure on $P(M)^n$. Let $S'$ be the imaginary sort of $P(M)^n$ modulo this equivalence.  If a theory $T$ is distal, so is $T^{\eq}$ by results from \cite{Distal}. Thus the $\tilde{\Cal M}$-induced structure on $S'$ is distal. However, $S'$ and $S$ are in in $M$-definable bijection in $\tilde{\Cal M}^{\eq}$. Thus the induced structure on $S$ from $\tilde{\Cal M}$ is distal. 

\end{proof}

\begin{thm}\label{DistalChar}
Let $x = (x_1,\ldots, x_n)$ be tuple of variables ranging over the home sort, and $p(x) \in S_x(M)$ be as in Lemma \ref{SmallSurj}. Further assume that $p(x)$ is $A$-invariant for some $A\subset M$ small relative to the saturation of $\Cal M$. Suppose that the induced structure on $P(M)$ from $\Cal M$ is distal.
Then:

\begin{enumerate}

\item If $p(x)$ is weakly orthogonal to $q(y)$, then it is distal.

\item If $P(N)$ is a dense $\Cal L$-elementary substructure of $\Cal N$ and $p(x)$ is distal, then $p(x)$ is weakly orthogonal to $q(y)$. 

\end{enumerate}
\end{thm}

\begin{proof}

We begin by proving (1). Let $p(x)$ be as in the theorem statement. Then by Lemma \ref{SmallSurj}, there is a small $M$-definable set $S \subseteq (a,b)$ such that $p(N) \subseteq S$. Now by Lemma \ref{InducedStructure}, the induced structure on $S$ is distal. Therefore the type $p(x)$ is distal. 

Now consider (2). We now further assume that $P(N)$ is a dense, proper $\Cal L$-elementary substructure of $\Cal N$. Suppose now that $p(x)$ is not weakly orthogonal to $q(y)$. Consider $c=(c_1,\ldots,c_n) \models p(x)$. For each $i \in \{1,\ldots n\}$, $c_i$ lies in no small $M$-definable set by Lemma \ref{SmallSurj}. Thus $\pi(c_1)\models q(y)$. It suffices to show that $\tp_{\Cal L_P}(c_1|M)$ is non-distal. We now assume that $\Cal N$ is $(2^{|M|})^+$-saturated, and by compactness and invariance of $p(x)$ find $(d_\alpha)_{\alpha \in (2^{|M|})^+}$ such that $d_\alpha \models \tp(c_1|M)|(M \cup \{d_\beta\}_{\beta<\alpha})$. Notice that by invariance of $p(x)$ we have that $\{d_\alpha\}_{\alpha < (2^{|M|})^+}$ is $\Cal L$-definably independent over $M \cup P(N)$. Otherwise some $d_\alpha$ would lie in an $M \cup \{d_\beta\}_{\beta<\alpha}$-definable small set. 

We now apply the Erd\"os-Rado Theorem to get an indiscernible sequence $(d'_j)_{j\in \Q + (e) +\Q}$ in $\Cal U$ such that for each $j_1<\ldots<j_m$ there are $\alpha_1<\ldots<\alpha_m$ such that $\tp(d'_{j_1}\ldots d'_{j_m}|M)=\tp(d_{\alpha_1}\ldots d_{\alpha_m}|M)$. Thus $\{d'_j:j\in \Q + (e) + \Q\}$ is also $\Cal L$-definably independent over $M\cup P(N)$. We now  show that the subsequence $\{d'_j:j\in \Q+\Q\}$ is $\pi(d'_{e})$-indiscernible. To do this, we again appeal to the back and forth system $\Gamma$. Let $i_1<\ldots<i_m$ and $j_1<\ldots<j_m$ be from $\Q+\Q$. Now consider $\Cal A_1'=\langle Md'_{i_1}\ldots d'_{i_m}\rangle$ and $\Cal A_2'=\langle Md'_{j_1}\ldots d'_{j_m}\rangle$. Notice that as the sequence $\{d'_j:j\in \Q+\Q\}$ is $M$-indiscernible and $\Cal L$-definably independent over $M \cup P(N)$, there is an isomorphism $g:A_1'\cong A_2'$ belonging to $\Gamma$ such that $g(d'_{i_k})=g(d'_{j_k})$ for $k \in \{1,\ldots m\}$. Now consider $\Cal A_i' \cap \pi^{-1}(\pi(d'_e))$. As $\{d'_j:j\in \Q+(e)+\Q\}$ is $\Cal L$-definably independent over $M \cup P(N)$, this intersection is empty. By saturation of $\Cal U$, there is $d''_e\in \pi^{-1}(\pi(d'_e))$ realizing the $\Cal L$-type at infinity over $\langle Md'_{i_1}\ldots d'_{i_m}d'_{j_1}\ldots d'_{j_m}\rangle$. Now consider $\Cal A_1=\langle Md'_{i_1}\ldots d'_{i_m}d''_e\rangle$ and $\Cal A_2 = \langle Md'_{j_1}\ldots d'_{j_m}d''_e\rangle$. By choice of $d''_e$, $g$ extends to an isomorphism $\hat{g}:\Cal A_1 \cong \Cal A_2$ with $\hat{g}(d'_{i_k})=\hat{g}(d'_{j_k})$ for each $k\in \{1,\ldots,m\}$ and $\hat{g}(d''_e)=\hat{g}(d''_e)$. Furthermore, as $d'_e$ was definably independent from $\{d_{i_1},\ldots d_{i_m},d_{j_1},\ldots d{j_m}\}$ over $M \cup P(N)$, so is $d''_e$ and thus $\hat{g}$ is in $\Gamma$. Thus $\hat{g}$ extends to an automorphism of $\Cal U^{\eq}$ fixing $\pi(d'_e)$ mapping $d'_{i_k}$ to $d'_{j_k}$ for each $k \in \{1,\ldots m\}$. As $i_1<\ldots<i_m$ and $j_1<\ldots<j_m$ were arbitrary, the sequence $\{d'_j:j\in \Q + \Q\}$ is $\pi(d'_e)$-indiscernible. Then, as $\pi(d'_j)=\pi(d'_e)$ holds if and only if $j=e$, the full sequence $\{d'_j:j\in \Q+(e)+\Q\}$ is not $\pi(d'_e)$-indiscernible, so $p(x)$ is non-distal. 
\end{proof}

\noindent We now consider some classes of examples where these results hold. In the case where $\Cal N \models T_P$ is a dense pair of expansions of o-minimal groups, it follows from Theorem 2 of \cite{DensePairs} that the induced structure on $P(N)$ is weakly o-minimal, and hence distal. Thus the results in this section and the previous completely characterize the invariant distal types, answering Question 2 positively. \newline

\noindent The second parts of both Lemma \ref{SmallSurj} and Theorem \ref{DistalChar} should not be considered as optimal. However, it appears that such results require some familiarity with a back and forth system for the desired concrete structures. For example, one may consider the structure of $(\R;+,\times,<,2^{\Q})$, or more generally an expansion of a real closed field by a dense multiplicative group with the Mann property. Our result with $a=0,b=\infty$ gives a generically stable type in the multiplicative quotient $(0,\infty)/2^{\Q}$. One can prove the analogues of the second parts of Lemma \ref{SmallSurj} and Theorem \ref{DistalChar}, and thus obtain a similar characterization of distal types. \newline

\noindent However, this result does not suffice in allowing one to study distal types in all of the non-distal theories from \cite{DNDP}. If $P(M)$ is a dense $\dcl_{\Cal L}$-independent set, then the resulting theory has elimination of imaginaries. Hence there is no such equivalence relation with dense, small classes for the construction in Section \ref{GenericExistence}. Thus there is no imaginary sort in $T_P^{\eq}$ to search for such a type $q(y)$. 

\section{Measures}\label{Measures}

We begin by recalling some definitions concerning Keisler measures. These definitions are not given in the fullest generality possible, but rather in the context we work in. In this section, fix a first order theory $T$ and $\Cal M\models T$. For $B\subseteq M$, by $\text{Def}_{M}(B)$, we mean the $B$-definable subsets of $M$. 

\begin{defn}

A \emph{Keisler Measure} $\mu$ on $M$ is a finitely additive probability measure on $\text{Def}_{M}(M)$. 

\end{defn}

\begin{remark}

Notice for any $\Cal N\succ \Cal M$ that $\text{Def}_M(N)$ can be  viewed as a boolean subalgebra of $\operatorname{Def}_{N}(N)$. General facts about measures then allow us to extend $\mu$ to $N$. 

\end{remark}

\noindent Due to this remark, for $M$-definable sets, we shall not distinguish between their measure considered as a subset of $N$ or as a subset of $M$. We now recall that this extension of measures is free, subject to minimal constraint. For a proof of the following result, see Theorem 7.4 of \cite{SimonBook}.

\begin{thm}\label{MeasureExt}

Let $\Cal N\succeq \Cal M$, $\mu$ be a Keisler measure on $M$, and $S \subseteq N$ be $N$-definable. Suppose that $$\sup\{\mu(S'): S'\subseteq S \text{ is in } \operatorname{Def}_M(N)\}\le r \le \inf\{\mu(S'):S \subseteq S' \text{ is in }\operatorname{Def}_M(N)\}$$
Then there is an extension of $\mu$ to $N$ such that $\mu(S)=r$. 

\end{thm}

\noindent So far we have been working with indiscernible sequences to study distality and distal types. We now introduce the relevant notions to define and study distal behavior through the lens of Keisler measures.

\begin{defn}

Let $\mu$ be a Keisler measure on $M$. We say $\mu$ is \emph{smooth} (over $M$) if for each $\Cal N \succeq \Cal M$ there is a unique extension of $\mu$ to $\text{Def}_N(N)$. 

\end{defn}

\noindent We again make a definition involving a monster model $\Cal U$. As before, this definition may be localized to not mention this monster model at the cost of quantifier complexity of the definition. 

\begin{defn}

Let $\mu$ be a Keisler measure on $U$. We say that $\mu$ is \emph{generically stable} if there is a small $\Cal M \prec \Cal U$ such that:

\begin{enumerate}

\item ($\mu$ is definable over $M$) For each $\epsilon > 0$ and definable $R \subseteq U \times U_y$, there is an $M$-definable partition $S_1,\ldots, S_n$ of $U_y$ such that for each $i \in \{1,\ldots,n\}$ and $b,b' \in S_i$, $|\mu(R(b))-\mu(R(b'))|<\epsilon$. 

\item ($\mu$ is finitely satisfiable in $M$) For each definable $S \subseteq \Cal U$, if $\mu(S)>0$ then $S \cap M \neq \emptyset$. 

\end{enumerate}

\end{defn}

\begin{remark}

\noindent Notice that if a Keisler measure $\mu$ over $M$ is definable, then for any $\Cal N\succeq \Cal M$ there is a canonical extension of $\mu$ to $N$. 

\end{remark}

\noindent We now can note an alternate definition for distality. A theory $T$ is distal if every generically stable measure (over some monster model) is smooth (see \cite{Distal}). We now continue with examining a specific generically stable, non-smooth measure.

\subsection{A Geometric Proof of Non-Distality in Certain Pairs}
Now, let $\Cal M = (M;<,+,0,1\ldots)$ be an o-minimal expansion of an ordered group in language $\Cal L$, and $P$ be subset satisfying the tameness conditions in \cite{PairStructure}. While these are the same assumptions as in Section \ref{GenericExistence}, we \emph{do not} assume the existence of an equivalence relation with dense, small classes. Then consider the dense pair $(\Cal M, P)$. We shall characterize the non-distality of this pair, working in the language $\Cal L_P = \Cal L \cup \{P\}$ where $P$ is interpreted as picking out the subset $P$. This determines a complete $\Cal L_P$-theory which we shall call $T_P$. \newline

\noindent In the case that $M = \mathbb R$, we can define a Keisler Measure on the $\Cal L_P(M)$-definable sets by $\mu(S) = \lambda(S \cap (0,1))$, where $\lambda$ is the Lebesgue measure. Now for any small $S\subseteq M$, $\mu(S) = 0$. Furthermore, if $0<a<b<1$ then $\mu((a,b))=b-a$. Notice that by Lemma \ref{PairDecomp}, this completely determines the measure of any definable $S \subseteq (0,1)$. \newline

\noindent We now generalize this to the case that $M \neq \mathbb R$. Fix $t \in M$ with $t>0$. For $a \in (0,t)$, we define $\text{st}(a) = \sup\{q\in \Q:qt\le a\}$. Notice that for any $a<b \in (0,t)$ with $a+b <t$ we have $\text{st}(a+b)=\text{st}(a)+\text{st}(b)$. We now define a Keisler measure on $M$ by setting $\mu(S)=0$ when $S\subseteq M$ is small, and for $0<a<b<t$, we set $\mu((a,b))=\st(b)-\st(a)$. Again by Lemma \ref{PairDecomp}, this uniquely determines a measure on all definable subsets of $M$. Throughout this section fix $\mu$ as this measure. 

\begin{remark}
Notice here that this construction in fact \emph{does not} require that $P$ be dense in all of $M$, but rather that it is dense in the interval $(0,t)$. We work in the context that $P$ is dense for ease of presentation.
\end{remark}

\begin{lem}
The measure $\mu$ is $\emptyset$-definable. 
\end{lem}

\begin{proof}
Let $R \subseteq M \times M^n$ be $\emptyset$-definable. We may assume that for each $c \in M^n$ that $R(x,c)\subseteq (0,t)$. Using Lemma \ref{PairDecomp}, we find $\emptyset$-definable functions $a_0,\ldots a_m$ such that for each $c \in M^n$, $0\le a_0(c)\le\ldots \le a_m(c)\le t$ and for $i \in \{1,\ldots m\}$, $[a_{i-1}(c),a_i(c)]\cap R(x,c)$ is either small or co-small. Denote this set as $R_i(x,c)$. Recall that for each $i$ the collection of $c \in M^n$ such that $[a_{i-1}(c),a_i(c)]\cap R(x,c)$ is co-small is $\emptyset$-definable. Notice then that $\mu(R(x,c)) = \sum_{i=1}^m \delta_i(c)\st(a_{i}(c)-a_{i-1}(c))$, where 

\[\delta_i(c) =\begin{cases}
0 & \text{if $R_i(x,c)$ is small}\\
1 & \text{if $R_i(x,c)$ is co-small}
\end{cases}\]

Let $S_c = \{i \in \{1,\ldots m\}:\delta_i=1\}$. Notice that the equivalence relation of $c \equiv c' \Leftrightarrow S_c= S_{c'}$ is $\emptyset$-definable. For ease of notation, we now assume there is one equivalence class and set $S_c = S$. Then for each $c \in M^n$, $$\mu(R(x,c)) = \sum_{i \in S}\st(a_{i}(c)-a_{i-1}(c)) = \st\left(\sum_{i \in S}a_{i}(c)-a_{i-1}(c)\right)$$

Now let a real number $\epsilon > 0$ be given. Choose $k \in \N$ such that $\frac{1}{k} < \frac{\epsilon}{2}$. Now for $j \in {1,\ldots, k}$, set $$B_j = \{c \in M^n: \sum_{i \in S}\left(a_{i}(c)-a_{i-1}(c)\right)\in [\frac{j-1}{k}t,\frac{j}{k}t]\}$$

Notice now that if $c \in B_j$, then $\mu(R(x,c)) \in [\frac{j-1}{k},\frac{j}{k}]$. Thus if $c,c' \in B_j$, we have that $|\mu(R(x,c))-\mu(R(x,c'))|<\epsilon$, and therefore $\mu$ is $\emptyset$-definable. 

\end{proof}

\begin{lem}

The measure $\mu$ is finitely satisfiable. 

\end{lem}

\begin{proof}

Let $\Cal N \succ \Cal M$ and set $\nu$ as $\mu | N$, the definable extension of $\mu$ to $N$. Let $S= \varphi(N,a)$ for some $\varphi \in \Cal L_P$ and $a \in N^n$. Suppose that $\nu(S)>0$. Furthermore, we may reduce to the case that $S \subseteq (0,t)_N$. Again using Lemma \ref{PairDecomp}, there are $M$-definable functions $a_0,\ldots, a_m$ such that for each $c \in N^n$, $0 \le a_0(c)\le \ldots \le a_m(c)\le t$, and for each $i \in \{1,\ldots m\}$, $[a_{i-1}(c),a_{i}(c)]_N\cap \varphi(N,c)$ is either small or co-small. As $\nu(S)>0$, for at least one such $i$ this intersection is co-small. Thus we may reduce to the case that $S$ is co-small. Suppose now for sake of contradiction that $S \cap M = \emptyset$. Then $M$ is small, as $M \subseteq (N \setminus S)$. Therefore there is an $a$-definable function $f:N^\ell \to N$ such that $M \subseteq f(P(N))^\ell$. Arguing similarly using $\dcl_{\Cal L}$-dimension as in the conclusion of Theorem \ref{TypeFinSat}, we reach a contradiction. 

\end{proof}

\noindent To complete showing the non-distality, it remains to show that $\mu$ is non-smooth. In the following result, similarly to as in Section \ref{GenericExistence}, we use $(a,b)_M = \{x \in M:a<x<b\}$ to denote the interval $(a,b)$ as a definable subset of $M$. 

\begin{thm}

The measure $\mu$ is non-smooth. 

\end{thm}

\begin{proof}

Let $\Cal N \succeq \Cal M$ be $|M|^+$-saturated. By saturation of $\Cal N$ and the smallness of $P(N)$, there is $b \in N$ that is $\Cal L$-definably independent from $M \cup P(N)$. Thus for every $M$-definable $f:N^\ell \to N$, $b \not \in f(P(N)^\ell)$. Furthermore, even $(b+P(N)) \cap f(P(N)^\ell))=\emptyset$, as otherwise $b \in f(P(N)^\ell)\setminus P(N)$.

Let $S \subseteq N$ be $M$-definable. Suppose that $S \subseteq b+P(N)$. As $S$ is $M$-definable, either $S = \emptyset$ or there is $a \in M \cap S$. In the second case $a \in S$ which is a subset of $b+P(N)$. This is a contradiction, as then $b+P(N)$ intersects the image of the constant function $x \mapsto a$. Therefore $S\cap M=\emptyset$, and $\mu(S)=0$.  

Now consider the case that $b+P(N)\subseteq S$ for some $M$-definable set $S$. By Lemma \ref{PairDecomp}, there are $a_0,\ldots a_m \in M\cup\{\pm \infty\}$ such that for each $i \in \{0,\ldots, m\}$, $-\infty\le a_0 \le \ldots \le a_m \le \infty$ and $S \cap [a_{i-1},\ldots a_i]$ is either small or co-small. Now since $\mu$ concentrates on $(0,t)$, by localizing to $(b+P(N)) \cap (0,t)$ and $S \cap (0,t)$, we may assume that $0\le a_0 \le \ldots \le a_m=t$. Suppose that for some $i \in \{1,\ldots m\}$ that $S \cap [a_{i-1},a_i]$ is small. By density of $P(N)$, $(b+P(N)) \cap (S \cap [a_{i-1},a_i]) \neq \emptyset$. As $S \cap [a_{i-1},a_i]$ is small and intersects $b+P(N)$, this contradicts the choice of $b$. Thus for each $i \in \{1,\ldots m\}$, $S \cap [a_{i-1},a_i]$ is co-small. Therefore $\mu(S)=1$.

We have now shown that if $S$ is an $M$-definable set, then $\mu(S)=0$ if $S \subseteq b+P(N)$ and $\mu(S)=1$ if $b+P(N) \subseteq S$. Therefore by Theorem \ref{MeasureExt}, for any $r \in (0,1)$ there is $\nu$, an extension of $\mu$ to $N$, such that $\nu(b+P(N))=r$. Thus $\mu$ is non-smooth.

\end{proof}

\noindent This gives another proof of Theorem 5.2 from \cite{DNDP}. We consider this proof ``geometric'' in that it relies only on the decomposition theorems for definable sets and properties of smallness in the expansions studied in \cite{PairStructure}.

\section{Structural Lemmas}\label{Structural Lemmas}

We shall now establish our context for the remaining sections, which focus on dense pairs of ordered vector spaces. Fix an ordered field $F$. Then let $T_{OVS}$ be the theory of ordered $F$-vector spaces in language $\Cal L = \{+,<,0,1,(x \mapsto \alpha x) \}_{\alpha \in F}$. The theory $T_{OVS}$ then states that $1$ is a distinguished positive element, $+$ and $<$ are addition and the ordering, and for each $\alpha \in F$, the symbol $\alpha$ is interpreted as scalar multiplication by that constant. \newline

\noindent Let $\Cal M = (M;+,<,0,1,(x\mapsto \alpha x)_{\alpha \in F}) \models T_{OVS}$. Then consider a proper dense $\Cal L$-substructure $\Cal Q$ of $\Cal M$. As $T_{OVS}$ admits elimination of quantifiers, this substructure is elementary. Therefore by \cite{DensePairs}, the $\Cal L \cup \{Q\}$-structure on $(\Cal M, \Cal Q)$ that interprets $Q$ as substructure membership determines a complete theory independent of the choice of $\Cal Q$. \newline

\noindent We do not work precisely in that structure on $(\Cal M, \Cal Q)$, but rather in a substructure of $(\Cal M, \Cal Q)^{\eq}$. In particular, we work in the two sorts of $M$ and $M/Q$, the sort of the $F$-vector space quotient of $M$ by $Q$. Let $\Cal L^\star = \Cal L \cup \{Q,\pi,+_Q,(x \mapsto\alpha_Q x)_{\alpha_Q \in F}\}$, where $\pi:M \to M/Q$ is the $F$-linear quotient map, $+_Q$ is the induced addition on $M/Q$, and each $\alpha_Q$ is multiplication by $\alpha_Q$ on $M/Q$. Then let $T_{POVS}$ be the theory of the $\Cal L^\star$-structure $(M,M/Q;+,<,0,1,(x \mapsto\alpha x)_{\alpha\in F},Q,\pi,+_Q,(x \mapsto\alpha_Q x)_{\alpha_Q \in F})$. Throughout we shall refer to the first sort as the home sort, and the second sort as the quotient sort. This raises no confusion for models of $T_{POVS}$, however for an arbitrary $\Cal L^\star$-structure $\Cal S=(S,S/Q;\ldots)$, it need not be the case that $S/Q$ is the quotient of $S$ by the substructure picked out by $Q$. \newline 

\noindent Notationally, for this section onward, we make a distinction between $x \in Q$ and $x \in Q(M)$. The first is to be taken as new notation for the $\Cal L^\star$-formula $Q(x)$, while the second states that $M \models Q(x)$. This second notation is thoroughly used when considering elementary extensions of models of $T_{POVS}$ later. 

\begin{thm}\label{QE}
The theory $T_{POVS}$ admits elimination of quantifiers in language $\Cal L^\star$. 
\end{thm}

\begin{proof}

\noindent Let $\Cal M, \Cal N \models T_{POVS}$, where $\Cal N$ is $|M|^+$-saturated. Consider a substructure $\Cal S =(S,S/Q,\ldots) \subseteq \Cal M$ with $\iota=(\iota_1,\iota_2):(S,S/Q) \to (N,N/Q)$ an $\Cal L^\star$-embedding. To establish elimination of quantifiers, it suffices to show that there there is an $\Cal L^\star$-embedding $\iota'=(\iota_1',\iota_2'):(M,M/Q) \to (N,N/Q)$ extending $\iota$.  We now proceed with the following two cases.

\begin{case}\label{OneStep}
For any $h \in M \setminus S$ with $\pi(h) \in S/Q$, there is an extension $(\iota_1',\iota_2')$ of $(\iota_1,\iota_2)$ such that $h \in \text{dom}(\iota_1')$. 
\end{case}

\begin{proof}[Proof of Case 1]
Let $h \in M \setminus S$ be such that $\pi(h) \in S/Q$. Denote $\pi(h)$ by $c$. We seek to find an $\Cal L^\star$-embedding $\iota'$ extending $\iota$ such that $h \in \text{dom}(\iota_1')$. Let $p(x)$ be the partial type of $<$-formulas satisfied by $h$ with parameters from $S$. Then let $\iota(p)$ be the collection of all formulas $\varphi(x,\iota(c))$ where $\varphi(x,c) \in p(x)$. Now $\iota(p)$ is a partial type with parameters from $N$. Notice that $\pi^{-1}(\iota_2(c))$ is dense in $N$. Therefore $\iota(p)\cup \{\pi(x)=\iota_2(c)\}$ is finitely satisfiable, and thus by saturation of $\Cal N$ is realized by some $k \in N$. Let $\Cal S\langle h \rangle$ be the $\Cal L^\star$-substructure of $\Cal M$ generated by $\Cal S$ and $h$. Set $\iota_1':\Cal S\langle h\rangle \to N$ as the $F$-linear map determined by $\iota_1$ along with $\iota_1'(h)=k$. By choice of $k$, $<$ is preserved by $\iota_1$. As the predicate $Q$ is an $F$-linear subspace, $Q$ is preserved by $\iota_1'$ also. Now $\pi(\iota_1'(h)) = \iota_2(\pi(h))$ as $k$ was chosen within $\pi^{-1}(\iota_2(c))$. As no new members of $S/Q$ were added, both $+_Q$ and each $\alpha_Q$ are also preserved. Therefore $(\iota_1',\iota_2):(\Cal S\langle h \rangle,S/Q)\to (N,N/Q)$ is an $\Cal L^\star$-embedding. 

\end{proof}

\begin{case}\label{Quotient Embedding}
Suppose $\pi(S) = S/Q$. Then there is an extension $(\iota_1',\iota_2')$ of $(\iota_1,\iota_2)$ such that $\text{dom}(\iota_2')= M/Q$.
\end{case}

\begin{proof}[Proof of Case 2]

\noindent Notice that $\iota_2:S/Q\to N/Q$ is an $F$-linear map. Applying quantifier elimination for $F$-vector spaces (see \cite[p. 383]{Hodges}), we get $\iota_2':M/Q\to N/Q$ an $F$-linear embedding. Therefore as no new elements of the home sort have been added, $(\iota_1,\iota_2'):(S,M/Q)\to (N,N/Q)$ preserves all $\Cal L^\star$-symbols except possibly for $\pi$. As $\pi(S) = S/Q$, $\text{dom}(\iota_2)\setminus (S/Q)$ is disjoint from $\pi(S)$, so $\pi(\iota_1(v)) = \iota_2(\pi(v))$ for each $v \in S$. 

\end{proof}

Using Case \ref{OneStep} as necessary, we may assume that $\pi(S) = S/Q$. Then by Case \ref{Quotient Embedding}, we may assume $S/Q = M/Q$. Then by further use of Case \ref{OneStep}, we get the desired $\Cal L^\star$-embedding $(\iota_1',\iota_2'):\Cal M \to \Cal N$ extending $(\iota,\iota')$. 

\end{proof}

\begin{remark}

Notice that as the extension on the quotient sort relies only on the $F$-vector space structure, the quotient sort of any model of $T_{POVS}$ is a pure $F$-vector space. 

\end{remark}

\begin{defn}
Let $\Cal M \models T_{POVS}$, and let $S$ be a non-empty definable subset of $M$. We say $S \subseteq M$ is a \emph{near-interval} if there are $a,b \in M \cup \{-\infty, \infty\}$, a finite $C \subseteq M/Q$, and $j \in \{-1,1\}$, such that $S = (a,b)\cap \pi^{-1}((M/Q)\setminus C)$ if $j=-1$ and $S = (a,b) \cap \pi^{-1}(C)$ if $j=1$. When $j=-1$ we call the near-interval \emph{large}, and when $j=1$ we call it \emph{small}. Notationally, we write $(a,b) \cap \pi^{-1}(C^j)$ for a near-interval. 
\end{defn}

\begin{remark}
First we notice that the usage of large and small agrees with the notion of $\pi^{-1}(0)$-large and small. Notice that the complement of a near-interval is a finite union of near-intervals and points, and that the intersection of two near-intervals is itself a near-interval. Also, by taking $C = \{\pi(1)\}$, $a=-\infty$, and $b=\infty$ we see that $Q(M)$ is a small near-interval.
\end{remark}

\begin{lem}\label{Unary Decomposition}

Let $\Cal M \models T_{POVS}$, and $S \subseteq M$ be definable. Then $S$ is a disjoint union of a finite set and finitely many disjoint near-intervals. 

\end{lem}

\begin{proof}

As $T_{POVS}$ admits quantifier elimination, we may restrict to sets defined by quantifier-free formulas. We shall call a formula basic if it is of the form $t(x)=0$, $t(x)>0$, $t(x) \in Q$, or $t(x) \not \in Q$, where $t(x)$ is some $\Cal L$-term. Now, any atomic formula with variables from the home sort is equivalent to a disjunction of basic formulas. Furthremore, this is also true of the negations of atomic formulas. Therefore, by passing to conjunctive normal form, any quantifier free formula $\varphi(x)$ may be written as $\bigwedge \psi_i(x)$ where each $\psi_i(x)$ is a disjunction of basic formulas. 

Note that a unary $\Cal L$-term with parameters from $M$ is of the form $t(x) = ax+b$ for some $a \in F$ and $b \in M$. Furthermore, as the reduct of $\Cal M$ to $\Cal L$ is o-minimal, formulas of the form $t(x) = 0$ and $t(x)>0$ define finite unions of points and intervals. Now consider a formula of the form $ax+b \in Q$. This is equivalent to $x \in (-1/a)b+Q$, which isolates a single coset of the predicate $Q(M)$. Similarly $ax+b \not \in Q$ isolates the complement of the coset $(-1/a)b+Q(M)$.  

Notice that we may write each $\psi_i(x)$ as $\varphi_1(x) \vee \varphi_2(x)$ where $\varphi_1(x)$ is an $\Cal L$-formula, and $\varphi_2(x)$ is a $(+,(x\mapsto \alpha x)_{\alpha \in F},Q)$-formula, but not an $\Cal L$-formula. The first defines a finite union of points and intervals in $M$. For the second, we consider a term $t(x)$ appearing in $\varphi_2(x)$. This term occurs either as $t(x) \in Q$ or $t(x) \not \in Q$. The first defines a single coset, the other the complement of that coset. The union of such sets will be either a finite union of cosets, the complement of a single coset, or all of $M$. Therefore $\varphi_2(x)$ isolates a set of that form, and $\varphi_1(x)\vee \varphi_2(x)$ is a finite union of points, intervals, and a single near-interval. Continuing thus, by enumerating these finitely many points and the endpoints of each interval, there is a partition $-\infty=a_0<\ldots<a_n=\infty$ such that for each $j \in \{0,\ldots, n-1\}$, $\psi_i((a_j,a_{j+1})$ is either empty or a near-interval. 

We now consider $\bigwedge \psi_i(x)$. For each $i$, let $-\infty=a_{i,0}<\ldots<a_{i,n_{i}}=\infty$ be a partition as in the previous paragraph. Interleaving each of these finitely many partitions gives a partition $-\infty = a_0<a_1<\ldots<a_N=\infty$ of $M$ such that for each $i$ and each $j \in \{0,\ldots, N-1\}$, $\psi_i((a_j,a_{j+1}))$ is either empty, or a near-interval. Thus on $S \cap (a_j,a_{j+1})$ is either empty or a finite intersection of near-intervals, which is itself a near-interval. The result now follows, as away from the finite set $\{a_1,\ldots, a_{N-1}\}$, $\bigwedge \psi_i(M)$ is a finite union of near-intervals. 

\end{proof}

\noindent As we now understand the structure of unary subsets of the home sort, we may make the following definition.

\begin{defn}

Let $\Cal M \models T_{POVS}$ and $S \subseteq M$ be definable. Then $S^{no}=\{x\in M: \exists \epsilon \in M_{>0} (x-\epsilon,x+\epsilon)\cap S \text{ is a near-interval}\}$. We then set the \emph{near-interior} of $S$ to be the intersection $S^{no}\cap S$, and the \emph{near-frontier} to be $S \setminus S^{no}$. 

\end{defn}

\begin{lem}\label{Near-Interior}

Let $\Cal M \models T_{POVS}$ and $S\subseteq M$ be definable, then both $S^{no}$ and the near-interior of $S$ are also definable.

\end{lem}

\begin{proof}

It remains to show that the property of $(x-\epsilon,x+\epsilon)\cap S$ being a near-interval is definable. By Lemma \ref{Unary Decomposition}, for any $a,b \in M$, $(a,b) \cap S$ is a finite union of points and near-intervals. Therefore $(a,b) \cap S$ is a near-interval if and only if $(a,b) \cap S = (a,b) \cap \pi^{-1}\left(\pi\left((a,b) \cap S\right)\right)$. 

\end{proof}

\begin{remark}
Notice that $S^{no}$ is always open. Now an open set is a union of basic open sets, which are open intervals. By Theorem \ref{Unary Decomposition}, we notice then that $S^{no}$ is, in fact, a union of finitely many intervals.  Furthermore, the endpoints of those intervals give a definable partition of the near-interior of $S$ into near-intervals. 
\end{remark}

\section{Distal Expansion}\label{Distal Expansion}

Fix $\Cal M \models T_{POVS}$. We seek an expansion of $\Cal M$ whose theory is distal. In \cite{DNDP}, the non-distality was witnessed by an indiscernible sequence $(a_i)_{i \in I_1 + (c) + I_2}$, where $I_1$ and $I_2$ are infinite linear orders without endpoints, $(a_i)_{i \in I_1+I_2}$ was indiscernible over some parameter $b$, but the full sequence was not. This construction yielded that $(a_i)_{i \in I_1+I_2}$ avoided the coset $b+Q(M)$, but $a_c$ fell into said coset. \newline

\noindent To avoid such sequences, we add a predicate $\prec$ such that $(M/Q;+_Q,\prec,(x\mapsto \alpha_Q x))_{\alpha_Q \in F}$ is a densely ordered $F$-vector space. We shall consider $\Cal M$ as an $\Cal L^\star \cup \{\prec\}$-structure, as an $\Cal L$-structure. And consider $M/Q$ as an $\Cal L_q = \{+_Q,\prec,(x \mapsto \alpha_Q x)_{\alpha_Q \in F}\}$-structure. By $0_Q$, we mean the zero element in the quotient sort.

\begin{thm}

The $\Cal L^\star \cup \{\prec\}$-theory of $\Cal M$ admits elimination of quantifiers.

\end{thm}

\begin{proof}

This is similar to Theorem \ref{QE}, except we use quantifier elimination for ordered $F$-vector spaces instead of just $F$-vector spaces in Case \ref{Quotient Embedding}.

\end{proof}

\begin{lem}\label{ExpTypeDecomp}

Let $p(x)$ be an $\Cal L^\star \cup \{\prec\}$-type over $M$. Then there is an $\Cal L$-type $p'(x)$, and an $\Cal L_q$-type $q(x)$ such that $p'(x)\cup q(\pi(x)) \proves p(x)$. Moreover $q(\pi(x))$ is over $M/Q$.  

\end{lem}

\begin{proof}

By quantifier elimination, it is enough to consider quantifier-free formulas. As $p(x)$ is a complete type, we may further restrict to atomic formulas. Let $\varphi(x,b)=\varphi(x_1,\ldots,x_m,b_1,\ldots,b_n) \in p$ be an atomic formula.  Then we may assume it is of the form $t(x,b)=0$, $t(x,b)>0$, $t(x,b)\in Q$, $t(x,b)\succ 0$, or $t(x,b)=0_Q$, where $t(x,y)$ is an appropriately chosen $\Cal L_\star \cup \{\prec\}$-term. 

Let us first consider $t(x,b)=0$, $t(x,b)>0$, and $t(x,b) \in Q$. In these, each variable $x_i$ must vary over the home sort and $b\in M^n$. The first two then are $\Cal L$-formulas, while the third is equivalent to $\pi(t(x,b)) \in Q$. As $\pi$ is $F$-linear, we may find a $\Cal L_q$-term $t'$ such that $t'(\pi(x),\pi(b)) = 0_Q$ if and only if $\pi(t(x,b)) \in Q$, which gives the desired $\Cal L_q$-formula.

Now the cases of $t(x,b) \succ 0_Q$ and $t(x,b)=0_Q$ remain. Here $x$ need not necessarily vary over $M^m$, but any occurrence of a variable in the home sort must occur within an instance of $\pi$. Furthermore, each instance of $\pi$ in $t(x,b)$ is of the form $\pi(t'(x,b))$ for some $\Cal L$-term $t'(x,b)$. As such a term is an $F$-affine function, and $\pi$ is an $F$-linear homomorphism, we may distribute $\pi$ over $t'$. This then gives an $F$-affine function over $M/Q$. Replacing each such subterm with variables from the home sort, we may now assume that $x$ varies over $(M/Q)^m$ and $b \in (M/Q)^n$. Then $t(x,b) \succ 0_Q$ and $t(x,b) = 0_Q$ are $\Cal L_q$-formulas. 

\end{proof}

\begin{thm}
The $\Cal L^\star \cup \{\prec\}$-theory of $\Cal M$ is distal.
\end{thm}

\begin{proof}

Recall that we need only check sequences within a single sort. We first note that as the structure on $M/Q$ is that of an ordered $F$-vector space, any indiscernible sequence in $M/Q$ is distal. We now consider sequences in the home sort.

Let $I_1$ and $I_2$ be infinite linear orders without endpoints and $(c)$ a one element linear order disjoint from $I_1$ and $I_2$. Let $b \in M$. Suppose that $(a_i)_{i\in I_1+(c)+I_2}$ is an indiscernible sequence in $M$ (it suffices to check unary sequences), and $(a_i)_{i \in I_1 + I_2}$ is $b$-indiscernible. We then notice that the $\Cal L$-structure on $\Cal M$ is that of an ordered $F$-vector space, and thus is distal. Therefore $(a_i)_{i\in I_1+(c)+I_2}$ is $\Cal L$-$b$-indiscernible. Now, the $\Cal L_q$-structure on $M/Q$ is that of an ordered $F$-vector space, which is distal. Therefore the sequence $(\pi(a_i))_{i \in I_1+(c)+I_2}$ is $\pi(b)$-indiscernible. By Lemma \ref{ExpTypeDecomp}, this is enough to guarantee that the sequence $(a_i)_{i \in I_1+(c)+I_2}$ is $\Cal L^\star \cup \{\prec\}$-indiscernible over $b$. Therefore the $\Cal L^\star \cup \{\prec\}$-theory of $\Cal M$ is distal. 
\end{proof}

\section{Elimination of Imaginaries}\label{Elimination of Imaginaries}

 We now give a weak elimination of imaginaries for models of $T_{POVS}$. This result was originally sought to try to provide a negative answer to Question 2 from the introduction, but it is now provided to show the essential uniqueness of the approach in Section \ref{Distal Expansion}, as long as one restricts to adding structure to quotient sorts. 

\begin{defn}

We say $T$ \emph{weakly eliminates imaginaries} if for each $\emptyset$-definable equivalence relation $E(x,y)$ and each model $\Cal M$ of $T$, there is some $\Cal L$-formula $\varphi(x,z)$ such that for each $a\in M$ there is a finite set $X_a$ such that $E(a,M) = \varphi(M,b)$ for each $b \in X_a$, and for $c \not \in X_a$, $E(a,M) \neq \varphi(M,c)$. 

We further say that this elimination is \emph{uniform} if there is an $\Cal L$-formula $\psi(x,z)$ such that $X_a = \psi(a,M)$ for each $a$. 

\end{defn}

\begin{defn}\label{wcodedef}

In a model $\Cal M$ of $T$, we say a $b$-definable set $S = \varphi(M,b)$ is \emph{weakly coded} if there is a finite $Y_b \subseteq M$ and an $\Cal L$-formula $\psi(x,y)$ such that $\psi(M,c) = \varphi(M,b)$ if and only if $c \in Y_b$. We say this coding is uniform if there is an $\Cal L$-formula $\chi(y,z)$ such that for any $b' \in M$, the set $Y_{b'} = \chi(M,b')$ is finite and weakly codes $\varphi(M,b')$ via $\psi(x,y)$. 

\end{defn}

\begin{remark}

\noindent Notice that in both previous definitions, uniformity allows us to check in a single model, and not in an arbitrary model of $T$. This is of particular use in Lemma \ref{criterion}.

\end{remark}

\noindent We note the following equivalence. 

\begin{lem}

The theory $T$ (uniformly) weakly eliminates imaginaries if and only if for each $\Cal M \models T$, $\Cal M$ (uniformly) weakly codes all definable subsets.

\end{lem}

\begin{proof}

\noindent The forward direction comes from setting $a E b$ iff $\varphi(M,a) = \varphi(M,b)$. The reverse comes from (uniformly) weakly coding the family $E(M,a)$. 

\end{proof}

\noindent We now establish a useful result to demonstrate the existence of weak codes. This lemma in proof and statement follows Lemma 14 of \cite{GenericEI}. 

\begin{lem}\label{criterion}

Let $\Cal M$ be a $|T|^+$-saturated model of $T$. Assume:

\begin{enumerate}

\item $M$ has two $\emptyset$-definable constants.

\item For each sort $s$, each subset of $M_s$ is uniformly weakly coded

\item For each tuple $s=(s_1,\ldots,s_n)$ of sorts and definable $R \subseteq M_s$, if $R(a)$ is finite for each $a \in M_{s_1}$, then $R$ is uniformly weakly coded. 

\end{enumerate}

Then $\Cal M$ uniformly weakly codes all definable subsets. 

\end{lem}

\begin{proof}

Let $S \subseteq M_s$ be $b$-definable, where $s=(s_1,\ldots,s_n)$ is an $n$-tuple of sorts. We proceed by induction on $n$. The case $n=1$ is true by $(2)$. 
For each $a \in M_{s_1}$, the fiber $S(a)$ may be uniformly weakly coded by the induction hypothesis. 
Thus there is an $\Cal L_b$-formula $\psi_a(x,y,z)$, a finite tuple of sorts $s_a$, and a finite set $X_a \subseteq M_{s_a}$, such that $\psi_a(M,c,a)=S(a)$ if and only if $c \in X_a$. By the uniformity $X_a = \varphi_a(M,a)$ for some $\Cal L_b$-formula $\varphi_a(y,z)$. As $M$ is $|T|^+$ saturated, there are $\Cal L_b$ formulas $\psi_1(x,y,z),\ldots \psi_n(x,y,z)$ with $y$ varying over a single tuple of sorts, along with $\varphi_1(y,z),\ldots, \varphi_n(y,z)$ such that for each $a \in M$, there are $1\le i,j \le n$ such that $\psi_i(x,y,z)$ and $\varphi_j(y,z)$ witness the uniform weak coding of $S(a)$. 
Now using the two $\emptyset$-definable constants, we may reduce to the case of single formulas $\psi(x,y,z)$ and $\varphi(y,z)$. Now for each $a \in M_{s_1}$, $\varphi(M,a)=X_a$ is finite. By assumption $(3)$, the $b$-definable relation $\varphi(M,M)$ can be uniformly uniquely coded as witnessed by $\chi(y,z,w)$ and the finite $b$-definable set $Y$. Notice that $\chi(M,M,d) = \varphi(M,M)$ if and only if $d \in Y$. Then $(e,f) \in S$ if and only if $\psi(f,c,e)$ holds for some $c$ in $X_a$. Now $c \in X_a$ if and only if $(c,a) \in \varphi(M,M)$, which holds if and only if $\chi(c,a,d)$ holds for some $d \in Y$. Thus $S$ is uniformly weakly coded, as the set $Y$ uniquely determines $S$, and is $b$-definable.  

\end{proof}

\noindent This section uses the same conventions as Section \ref{Structural Lemmas}. We now show a uniform weak elimination of imaginaries for models of $T_{POVS}$. While pursuing the results in Section \ref{Distal Expansion}, we found it helpful to know what precisely the imaginary sorts are, even though it was unnecessary for the main result. This can be considered as showing that adding structure on the vector space quotient is essentially the only way to get a distal expansion of $T_{POVS}$ by expanding $T_{POVS}^{\eq}$. We proceed now to weakly eliminate imaginaries by establishing that the assumptions of Lemma \ref{criterion} are satisfied.

\begin{prop}\label{weakcodes}
Let $\Cal M \models T_{POVS}$, and $S\subseteq M$ be definable. Then $S \subseteq M$ is uniformly weakly coded.
\end{prop}

\begin{proof}

By Lemma \ref{Unary Decomposition}, we know that $S$ is a disjoint union of a finite set and finitely many near-intervals. We may assume that finitely many near-intervals form the near-interior of $S$, which is definable by Lemma \ref{Near-Interior}. Recall that then there are $n\ge 0$ and $-\infty\le a_0 < b_0 \le \ldots \le a_n<b_n\le \infty$, definable from $S$, such that the near-interior of $S$ is $\bigsqcup_{i=0}^{n}\left((a_i,b_i)\cap S\right)$. Therefore for each $i \in \{0,\ldots n\}$ there are $C_i \subseteq_{\text{fin}} M/Q$ and $j(i) \in \{-1,1\}$ such that $(a_i,b_i) \cap S = (a_i,b_i) \cap \pi^{-1}(C_i^{j(i)})$. 

Let $\{s_1,\ldots, s_m\}$ enumerate the near-frontier of $S$ in increasing order.  Let $C_i^{en}$ be the set of all enumerations of $C_i$. Then set $Y = \{(s_1,\ldots s_m,a_0,b_0,c_0,\ldots,a_n,b_n,c_n):c_i \in C_i^{en}\}$. It is clear that from $S$ we may determine the set $Y$. Any member of $Y$ then also determines $S$, as $S$ is the unique set whose near-frontier is $s_1<\ldots<s_n$, $S^{no} = \bigsqcup_{i=0}^n (a_i,b_i)$, and $\pi((a_i,b_i) \cap S) = C_i^{j(i)}$. 
\end{proof}

Semi-uniform weak elimination of imaginaries, a weaker notion than uniform weak elimination, holds in $F$-vector spaces, and thus in the sort $M/Q$ \cite[p. 161]{Hodges}. However, a gluing argument using the $\emptyset$-definable constants in the home sort allows us to make this elimination uniform. For more details, please see Lemma 4.4.5 of \cite{Hodges}.

\begin{prop}\label{induction step}
Let $\Cal M \models T_{POVS}$ be $|T|^+$-saturated, and let $R(x,y,z) \subseteq M \times M^{n_1}\times (M/Q)^{n_2}$ be $\emptyset$-definable. Let $a \in M$, and assume each fiber $R(a,y,z) \subseteq M^{n_1} \times (M/Q)^{n_2}$ is finite. Then $R$ is uniformly weakly coded. 
\end{prop}

\noindent We proceed now with a series of lemmas. Recalling the remark after definition \ref{wcodedef}, we may assume that $\Cal M$ is $\aleph_1$-saturated.

\begin{lem}\label{functionreduction}

Let $\Cal M$ and $R$ be as in \ref{induction step}. Further assume that $n_2=0$, that is that $R(a,y)$ lives purely in the home sort for each $a \in M$. Suppose that for each definable partial function $f:M \to M$, that $\text{gr}(f)$ can be uniformly weakly coded. Then $R(x,y) \subseteq M \times M^{n_1}$ can be uniformly weakly coded. 

\end{lem}

\begin{proof}

By the saturation of $\Cal M$, we may take $N \in \N$ such that $|R(a)|\le N$ for each $a \in M$. Furthermore, assume that $|R(a,y)|=N$ for some $a \in M$. For $i \in \{1,\ldots N\}$ and $j \in \{1,\ldots n_1\}$, define the partial function $f_{i,j}(a)$ to be the $j$-th entry in the $i$-th element of $R(a,y)$ under the lexicographic order, when this exists. By assumption, each $f_{i,j}$ can be uniformly weakly coded. However, the collection of the $f_{i,j}$ is uniquely determine the relation $R(x,y)$, as $R(x,y)$ is the unique subset of $M \times M^{n_1}$ such that the $j$-the element of the $i$-th lexicographically ordered element of the fiber above $a$ is given by $f_{i,j}(a)$. 

\end{proof}

\noindent We now show that each partial function from $M$ to $M$ can be uniformly weakly coded. 

\begin{lem}\label{functioncode}

Let $S\subseteq M$ and $f:S \to M$ be definable. Then $\text{gr}(f)$ can be uniformly weakly coded. 

\end{lem}

\begin{proof}

In the case that $S$ is finite, then $\text{gr}(f)$ is finite, and can clearly be uniformly weakly coded. Therefore we may assume that $S$ contains a near-interval. 

Now, $T_{POVS}$ has no dense graphs by 5.9 in \cite{OpenCore}. Therefore we may use the remark following Lemma 16 in \cite{GenericEI}, which yields that $\overline{\text{gr}(f)}$ is a finite union of $\Cal L$-definable functions. Each $\Cal L$-definable unary function is of the form $x\mapsto \alpha x + b$ for some $\alpha \in F$ and $b \in M$. Therefore there are $\alpha_1<\ldots<\alpha_n \in F$ and $b_1,\ldots b_m \in M$ such that for each $x \in S$, $f(x)=\alpha_ix+b_j$ for some $i \in \{1,\ldots n\}$ and $j \in \{1,\ldots m\}$. 

Let $i \in \{1,\ldots n\}$. For $x$ in the near-interior of $S$, we say $x$ is an $\alpha_i$-point if there is $\epsilon>0$ such that for all $y \in (x-\epsilon,x+\epsilon) \cap S$, $f(y)-\alpha_iy = f(x)-\alpha_ix$. Let $S_i$ be the collection of $\alpha_i$-points. Each non-empty $S_i$ is a finite union of near-intervals on which $f(x)-\alpha_ix$ is constant $b_j$ for some $j \in \{1,\ldots m\}$. Enumerate these distinct constants as $c_{i,1}<\ldots< c_{i,\ell_i}$. Set $U_{i,k}$ be the subset of $S_i$ where $f(x)-\alpha_ix = c_{i,k}$. Notice that $U_{i,k}\cap U_{i',k'} \neq \emptyset$ only if $i=i'$ and $k=k'$.

We now claim that $S \setminus \bigcup_{i=1}^n S_i$ is finite. If not, there are $a,b \in M$ such that $(a,b) \cap S$ is a near-interval. Therefore there are $i,j$, such that $\{x \in (a,b)\cap S:f(x)=\alpha_ix+b_j\}$ is infinite. This then contains a near-interval, which is necessarily a subset of $S_i$, contradicting the choice of $a$ and $b$.

For ease of notation, let us assume each $S_i$ is non-empty. In the case that one is empty, we must be careful not to include it in the following. Now each $U_{i,k}$ can be uniformly weakly coded by Proposition \ref{weakcodes}. Notice that $f$ is uniquely determined by its values on the finite set $S \setminus \bigcup_{i=1}^n S_i$ along with the fact that $f(x) = \alpha_ix+c_{i,k}$ on each $U_{i,k}$. Thus $f$ can be uniformly weakly coded. 

\end{proof}

\noindent Now we conclude with proving Proposition \ref{induction step}.

\begin{proof}

In the previous lemmas, we uniformly weakly coded the projection of $R(x,y,z)$ onto $M \times M^{n_1}$. Now for each $a \in M$ and $b=(b_1,\ldots b_{n_1})\in M^{n_1}$, we know that $R(a,y,z) \subseteq M^{n_1} \times (M/Q)^{n_2}$ is finite. Thus $R(a,b_0,\ldots,b,z)$ is a finite $a$-definable set. We may assume that $R(x,y,z)$ is quantifier-free definable. Furthermore, we may arrange the formula defining $R(x,y,z)$ to be in disjunctive normal form. Write this formula as $\varphi(x,y,z) = \bigvee \varphi_i(x,y,z)$, where $(x,y)$ varies over $M^{1+n_1}$, $z=(z_1,\ldots,z_{n_2})$ varies over $(M/Q)^{n_2}$, and each $\varphi_i(x,y,z)$ is a conjunction of atomic formulas. As $R(a,b)$ is finite for each $(a,b) \in M^{1+n_1}$, we may further assume that $\varphi_i(a,b, M) \cap \varphi_j(a,b, M) = \emptyset$ for each $i \neq j$. 

Now $R(a,y,z)$ is finite for each $a \in M$. As $M/Q$ carries the structure of a pure $F$-vector space, for each $b \in M^{n_1}$, $R(a,b,z)$ is a finite union of affine $F$-subspaces. Furthermore, each such affine $F$-subspace is $0$-dimensional. Therefore, if $\varphi_i(a,b,z)$ is nonempty, it is a singleton. Considering the possible quantifier-free $\Cal L^\star$-formulas, we then see that $\varphi_i(a,b, z)$ can be written as $L(\pi(a,b))=z$ for some $F$-linear function $L$. Thus for each $a \in M$ there are $F$-linear functions $L_{1,a},\ldots,L_{m_a,a}$ such that $R(a,b,z) = \bigcup L_{i,a}(b)$. As $\varphi(x,y,z)$ is of finite length, the collection of all the $L_{i,a}$ is finite. Thus there are $L_1,\ldots, L_m$ such that for each $(a,b) \in M^{1+n_1}$ and $c \in R(a,b,z)$, there is $i \in \{1,\ldots m\}$ such that $c = \pi(L_i(a,b))$. Let $i \in \{1,\ldots,n_1\}$. Outside a finite subset of $M^{1+n_1}$, the proofs of Lemmas \ref{functionreduction} and \ref{functioncode} show that each $b_\ell$ is a linear function of $a_0$. Thus, by substituting in these linear functions, we may assume each $L_j$ is a unary function. That is, that for each $(a,b)=(a,b_0\ldots,b_{n_1}) \in M^{1+n_1}$ and $c \in R(a,b,z)$, that there is $j \in \{1,\ldots m\}$ such that $c = \pi(L_j(a))$. 

We now weakly code the relation $R(x,y,z)$. For each $S \subseteq\{1,\ldots m\}$, we may determine whether $R(a,b,z) = \{\pi(L_j(a_0)):j \in S\}$. This gives a definable partition of $M^{1+n_1}$. Each piece of this partition can be uniformly weakly coded by Lemmas \ref{functionreduction} and \ref{functioncode}. Furthermore, the collections of functions used on each partition can be uniformly weakly coded by adapting Lemma \ref{functioncode}. This then gives a uniform weak coding of the relation $R(x,y,z)$. 

\end{proof}

\noindent We have now established the assumptions of Lemma \ref{criterion} except in the case where the sort $s_1$ in $(3)$ is the quotient sort. Given $R \subseteq M/Q \times M^{n_1} \times (M/Q)^{n_2}$ with $R(a)$ finite for each $a \in M/Q$, define $\hat{R}\subseteq M \times M^{n_1}\times (M/Q)^{n_2}$ such that $\hat{R}(a) = R(\pi(a))$. Notice $\hat{R}(a)$ is finite for each $a \in M$, and thus can be uniformly weakly coded. Notice then that $\hat{R}$ uniquely determines $R$. Therefore the assumptions of Lemma \ref{criterion} are met. 

\begin{cor}

Let $\Cal M \models T_{POVS}$. Every definable $S \subseteq M$ can be uniformly weakly coded. Therefore $T_{POVS}$ uniformly weakly eliminates imaginaries.  

\end{cor}

\end{document}